\documentclass[oneside, 10pt]{amsart}
\usepackage{amscd, amsmath, amssymb, amsthm, amsfonts, amstext, geometry, verbatim, enumitem, graphicx, mathtools, xfrac, microtype, nameref, thmtools}
\usepackage[breaklinks=true]{hyperref}
\usepackage[capitalize]{cleveref}
\usepackage[hyperref=true, backend=bibtex, firstinits=true, citestyle=numeric-comp, sortlocale=en_US, url=false, doi=false, eprint=true, maxbibnames=4]{biblatex}            
\usepackage[matrix,arrow,curve]{xy}
\usepackage[toc,page]{appendix}

\renewbibmacro*{volume+number+eid}{\ifentrytype{article}{\- \iffieldundef{volume}{}{Vol.~\printfield{volume},}\iffieldundef{number}{}{ No.~\printfield{number},}}}
\renewbibmacro{in:}{\ifentrytype{article}{}{\printtext{\bibstring{in}\intitlepunct}}}
\newbibmacro{string+doi}[1]{\iffieldundef{doi}{\iffieldundef{url}{#1}{\href{\thefield{url}}{#1}}}{\href{http://dx.doi.org/\thefield{doi}}{#1}}}
\DeclareFieldFormat[article, inproceedings, inbook, book, thesis]{title}{\usebibmacro{string+doi}{\mkbibquote{#1}}}

\newlist{lemlist}{enumerate}{1} \setlist[lemlist]{label={\rm(\arabic{lemlisti})}, ref=\thelemma.(\arabic{lemlisti}),noitemsep} \Crefname{lemlisti}{Lemma}{Lemmas}

\theoremstyle{plain}

\newtheorem{thm}{Theorem}
\Crefname{thm}{Theorem}{Theorems}
\numberwithin{equation}{section}

\newtheorem{lemma}{Lemma}
\numberwithin{lemma}{section}
\Crefname{lemma}{Lemma}{Lemmas}

\newtheorem{cor}[lemma]{Corollary}
\Crefname{cor}{Corollary}{Corollaries}

\newtheorem{prop}[lemma]{Proposition}
\Crefname{prop}{Proposition}{Propositions}

\newtheorem*{thm*}{Theorem}
\newtheorem*{lemma*}{Lemma}

\theoremstyle{remark}

\newtheorem{rem}[lemma]{Remark}
\Crefname{rem}{Remark}{Remarks}

\theoremstyle{definition}

\newtheorem{df}[lemma]{Definition} \Crefname{df}{Definition}{Definitions}
 \Crefname{example}{Example}{Examples}

\newcommand{\catname}[1]{{\normalfont\textbf{#1}}}
\DeclareMathOperator{\HH}{H}

\DeclareMathOperator{\Ker}{Ker}
\DeclareMathOperator{\Hom}{Hom}
\DeclareMathOperator{\sd}{sd}

\newcommand{\ZZ}{\mathbb{Z}}

\newcommand{\pullbackcorner}[1][dr]{\save*!/#1-1.2pc/#1:(-1,1)@^{|-}\restore}

\title[Parametrized symmetric groups]{Parametrized symmetric groups and \\ the second homology of a group}
\author {Sergey Sinchuk}
\address{Chebyshev laboratory, St. Petersburg State University, St. Petersburg, Russia}
\email {sinchukss at gmail.com}
\date {\today}

\begin{document}

\begin{abstract}
We introduce the notion of a symmetric group parametrized by elements of a group.
We show that this group is an extension of a subgroup of the wreath product $G \wr S_n$ by $\HH_2(G, \ZZ)$.
We also discuss the motivation behind this construction.
\end{abstract}

\maketitle

\section{Introduction}
The aim of this note is to describe a connection between symmetric groups and nonabelian tensors products 
 which apparently has not been studied in detail previously.
 
Let $G$ be a (possibly noncommutative) group. 
Denote by $S_n(G)$ the group generated by elements $s_i(a)$, where $i$ is an integer $1\leq i\leq n-1$ and $a$ is an element of $G$, subject to the following relations:
\begin{align}
s_i(a)^2                 = &\, 1,                                     & a \in G,                  \label{Cox1} \\
s_i(a) s_{i+1}(b) s_i(c) = &\, s_{i+1}(a^{-1}cb) s_i(a) s_{i+1}(b),   & a, b, c \in G,            \label{Cox2}  \\
s_i(a) s_j(b)            = &\, s_j(b) s_i(a),                         & |i-j|\geq 2,\, a, b\in G. \label{Cox3} 
\end{align}
Notice that in the special case $G=1$ the above presentation boils down to the usual Coxeter presentation of the symmetric group $S_n$,
 i.\,e. $S_n(1) = S_n$.
 
The group defined above bears some similarity with the group $G_n(R)$ defined by C.~Kassel and C.~Reutenauer in~\cite{KR98}
 (the generators of this group are parametrized by elements of a ring $R$ rather than a group, cf. also with the {\it parametrized braid group} defined by J.~-L.~Loday and M.~Stein in~\cite{LS05}).
By analogy, we call the group $S_n(G)$ a {\it parametrized symmetric group}.
Our first main result is the following theorem, which provides its concrete description.

\begin{thm} \label{thm:summary} For $n \geq 3$ there is a map $\mu_n\colon S_n(G) \to G \wr S_n$ which is a crossed module.
The cokernel and the kernel of $\mu_n$ are isomorphic to $\HH_1(G, \ZZ)$ and $\HH_2(G, \ZZ)$ respectively.
\end{thm}
In fact, we compute $S_n(G)$ in terms of so-called extensions of type $\mathfrak{H_n}(G)$, see~\cref{cor:main}.
These groups were introduced by U.~Rehmann in~\cite{Reh78} in the context of noncommutative Matsumoto theorem.
They are directly related with nonabelian tensor products.
For example, the nonabelian antisymmetric square $G \mathbin{\widetilde{\otimes}} G$ is a subgroup of the group $H_n(G)$, the ''universal`` extension of type $\mathfrak{H}_n(G)$
 (we refer the reader to Section~\ref{sec:Hnextensions} for more details).

 
In course of the proof of~\cref{thm:summary} we find a presentation of the universal extension of type $\mathfrak{H}_n(G)$ which is
 simpler than its original presentation given in~\cite{Reh78} (4 instead of 5 relations, which are also shorter), see~\cref{prop:simpler}.

Parametrized symmetric groups occur as fundamental groups of certain spaces, which, in turn,
 were discovered out of an attempt to answer some purely topological question. 
We describe this question below in more detail.

For a pointed simplicial set $(X, x_0)$ denote by $\ZZ[X]$ the associated free simplicial abelian group.
Define the ''alternating`` map of simplicial sets $h_n\colon X^{2n} \to \ZZ[X]$ by the following formula:
 \[h_n(x_1, x_2, \ldots x_{2n}) = \sum\limits_{i=1}^{2n}(-1)^ix_i.\]
Define the ''stabilization`` map $X^{2n} \hookrightarrow X^{2(n+1)}$ by $x \mapsto (x, x_0, x_0)$ and denote the colimit of $X^n$ by $X^\infty$.

Motivated by the study of homotopy invariants of maps, S.~Podkorytov showed for a connected simplicial set $X$ that the limit map 
$h_\infty = \mathrm{colim}_n(h_n) \colon X^\infty \to \ZZ[X]$ is a quasifibration.
This means, in particular, that the natural map $h^{-1}_\infty(b) \to F_{h_\infty}(b)$ between the fiber and the homotopy fiber of $h_\infty$ over any point $b \in \ZZ[X]_0$ is a weak equivalence.
This assertion essentially follows from~\cite[Lemma~9.1]{Po17}, although some technical work is required to deduce this. 

One might ask whether the fibers of $h_n$ satisfy some weaker property.
For example, one might expect that the map $h_n^{-1}(0) \to F_{h_n}(0)$ is $c(n)$-connected for some $c(n)$ which depends only on $n$ and tends to infinity as $n\to \infty$.
  
The following result which is a consequence of Theorem 1 asserts that this map at least induces an isomorphism of fundamental groups in one particular special case.
\begin{thm} \label{thm:main} If $X=BG$ is the classifying space of a group $G$ then for $n\geq 3$
 the natural map $h_n^{-1}(0) \to F_{h_n}(0)$ induces an isomorphism of fundamental groups
 $\pi_1(h^{-1}_n(0))\cong \pi_1(F_{h_n}(0))$. \end{thm}
Apart from~\cref{thm:summary} the proof of \cref{thm:main} involves some explicit computations with simplicial fundamental groups 
 using Quillen Theorem A and van Kampen theorem.

\subsection{Acknowledgements}
I wish to thank S.~S. Podkorytov for posing the problem and numerous helpful discussions.
I am also grateful to S.~O.~Ivanov, V.~Isaev and N.~A.~Vavilov for their valuable comments and interest in this work.
During the final stage of this work the author was supported by the Russian Science Foundation grant 19-71-30002.

\section{Definition of $S_n(G)$ and its basic properties} \label{sec:QnG-def}
The aim of this section is to prove the first part of~\cref{thm:summary}, namely the fact that $S_n(G)$ is a crossed module over $G \wr S_n$.
It will be useful for us to start with another definition of the group $S_n(G)$ given in terms of amalgams of symmetric groups.
In the end of the section we show that this alternative definition coincides with the one given in the introduction.

Let $G$ be a group. Recall that the {\it wreath product} is, by definition,
the semidirect product $G^n \rtimes S_n$, in which $S_n$ acts on $G^n$ on the right by permuting its factors.

Consider the family $\{{S_n}^{(g)}\}_{g\in G^n}$ of isomorphic copies of $S_n$ indexed by elements of $G^n$ and let $F$ be the free product of groups from this family.
For $s\in S_n$ and $g\in G^n$ we denote by 
$s_{g}$ the image of $s$ in $F$ under the canonical map $S_n^{(g)} \to F$.
Define the group $S_n(G)$ as the quotient of $F$ modulo the following single family of relations:
\begin{equation} \label{eq:main_rel} s_{g} = s_{h}, \text{ if $s^g = s^h$ holds in $G \wr S_n$.} \end{equation}
We still denote by $s_g$ the image of the element $s_g\in F$ under the canonical map $F\to S_n(G)$.

\begin{rem}
Expanding the definition of the semidirect product we get that
\begin{equation} \nonumber s^g = (1_{G^n}, s)^{(g, 1)} = (g^{-1}, s) (g, 1) = (g^{-1} g^{(s^{-1})}, s). \end{equation}
therefore the equality $s^g = s^h$ holds in $G \wr S_n$ iff
$g^{-1} g^{(s^{-1})} = h^{-1} h^{(s^{-1})}$, or what is the same, iff $s$ fixes $hg^{-1}$.

The last statement immediately implies that the map $\mu \colon S_n(G) \to G \wr S_n$ given by $\mu(s_g) = s^g = (g^{-1}g^{(s^{-1})}, s)$ is well-defined.
\end{rem}

Observe from the definition of $S_n(G)$ that there is a split exact sequence.
\begin{equation} \label{eq:ex-seq} \xymatrix{1 \ar[r] & \mathrm{Ker}(\pi) \ar[r] &  S_n(G) \ar@<0.5ex>[r]^{\pi} & S_n \ar[r] \ar@{-->}@<0.5ex>[l]^{\iota(1)} & 1,} \end{equation}
Here the map $\pi = \pi_{S_n} \circ \mu$ removes subscript $g$ from each $s_g$ and the section $\iota(1)$ maps $s$ to $s_1$, where $1$ is the identity element of $G^n$.
Clearly, $S_n(G)$ decomposes as $HS_n(G) \rtimes S_n$, where $HS_n(G) = \mathrm{Ker}(\pi)$.

\subsection{Crossed module structure on $S_n(G)$}
Recall from~\cite[\S~2.2]{BHS11} that a {\it crossed module} is a morphism of groups $\mu\colon M\to N$ together with a right action of $N$ on $M$ 
compatible with the conjugation action of $N$ on itself, i.\,e.
\begin{equation} \label{eq:precrossed} \tag{CM1} \mu(m^n) = \mu(m)^n \text{ for all $n \in N$, $m \in M$}, \end{equation}
which also satisfies the following identity called {\it Peiffer identity}:
\begin{equation} \label{eq:Peiffer} \tag{CM2} m^{m'} = m^{\mu(m')} \text{ for all $m, m' \in M$}.\end{equation}
In our situation, we let $G \wr S_n$ act on $S_n(G)$ by 
\begin{equation} \label{eq:action} (s_g)^{(h, t)} = {s^t}_{(gh)^t}, \text{ for $s, t \in S_n$ and $g, h \in G^n$.} \end{equation}
The goal of this subsection is to prove the following.
\begin{prop} \label{thm:cms} For $n \geq 3$ the map $\mu \colon S_n(G) \to G \wr S_n $ is a crossed module. \end{prop}
From the fact that $\mu$ is a crossed module one can deduce that $\Ker(\mu)$ is a central subgroup of $S_n(G)$ and $\mathrm{Im}(\mu)$ is a normal subgroup of $G \wr S_n$.
It is not hard to show that the group $S_2(G)$ is isomorphic to a free product of copies of $S_2$ indexed by elements of $G$ 
 (the center of this group is trivial if $G\neq 1$).
Therefore, the requirement $n\geq 3$ in the statement of~\cref{thm:cms} is essential.

Verification of the fact that formula~\eqref{eq:action} gives a well-defined action of $G\wr S_n$ on $S_n(G)$ 
 that satisfies~\eqref{eq:precrossed} is lengthy but straightforward. Let us show that~\eqref{eq:Peiffer} holds. 
It suffices to verify Peiffer identities only for the generators of $S_n(G)$, for which it takes the form:
\begin{equation} \label{eq:Peiffer-gen} t^{-1}_h s_g t_h = {s^t}_{(gh^{-1})^t\cdot h} \text{ for all $s, t \in S_n$, $g, h\in G^n$.}\end{equation}
If we act on both sides of the above formula by $(h^{-1}, 1) \in G \wr S_n$ we obtain the equality
$t^{-1}_1 s_{gh^{-1}} t_1 = {s^t}_{(gh^{-1})^t}$.
Thus, to prove~\eqref{eq:Peiffer-gen} it suffices to show the following simpler relation:
\begin{equation} \label{eq:Peiffer-simple} t^{-1}_1 s_g t_1 = {s^t}_{g^t} \text{ for all $s, t \in S_n$, $g\in G^n$.}\end{equation}

The key step in the proof is the following lemma.
\begin{lemma} \label{lem:transp-deff} 
 The relation~\eqref{eq:Peiffer-simple} holds in the special case when $s=(ij)$ and $t=(kl)$ are two nonequal transpositions. \end{lemma}
\begin{proof} 
First of all, we immediately check that~\eqref{eq:Peiffer-simple} holds in the special case when $t$ fixes $g'\in G^n$.
Indeed, by~\eqref{eq:main_rel} we have $t_1 = t_{g'}$, hence 
\begin{equation} \nonumber t^{-1}_1 s_{g'} t_1 = t^{-1}_{g'} s_{g'} t_{g'} = (s^t)_{g'} = (s^t)_{g'^t}.  \end{equation}

Without loss of generality we may assume that $l\neq i$ and $l\neq j$.
Denote by $g'$ the vector which differs from $g$ only at $l$-th position, for which we set $g'_l = g_k$. 
Since the only nontrivial component of $g'g^{-1}$ (resp. $g'g^{-t}$) is located at $l$-th (resp. $k$-th) position,
it is fixed by $s$ (resp. $s^t)$, hence from~\eqref{eq:main_rel} we conclude that $s_g = s_{g'}$ (resp. $(s^t)_{g'} = (s^t)_{g^t}$).
Finally, since $g'$ is fixed by $t$, we get that
\begin{equation} \nonumber t^{-1}_1 s_g t_1 = t^{-1}_1 s_{g'} t_1 = (s^t)_{g'^t} = s^t_{g'} = (s^t)_{g^t}. \text{\qedhere}\end{equation} \end{proof}

\begin{proof}[Proof of~\cref{thm:cms}]
Let us show that~\eqref{eq:Peiffer-simple} holds for arbitrary transpositions $s, t \in S_n$.
It suffices to consider the case $s=t=(ij)$. 
After choosing some $k\neq i,j$ and presenting $(ij)$ as $(kj)(ik)(kj)$ we use \cref{lem:transp-deff}:
\begin{equation} \nonumber (ij)^{-1}_1 (ij)_g (ij)_1 = (ij)^{-1}_1 (kj)_g (ik)_g (kj)_g (ij)_1 =
(ki)_{g^{(ij)}} (jk)_{g^{(ij)}} (ki)_{g^{(ij)}} = (ij)_{g^{(ij)}}. \end{equation}
The proposition now follows by induction on the length of permutations $s$, $t$.
\end{proof}

\subsection{An explicit presentation of $S_n(G)$}
In this section we obtain an explicit presentation of $S_n(G)$ in terms of reflections.
From this presentation we later deduce both the Coxeter presentation for $S_n(G)$ and
 an explicit presentation for the subgroup $HS_n(G)$.

We start with the following simple lemma which is, in essence, a variant of the standard presentation of the symmetric group in terms of reflections
 with some redundant generators and relations added (cf.~\cite[Theorem~2.4.3]{Ca89}).
\begin{lemma} \label{lm:Snpres} For $n\geq 3$ The symmetric group $S_n$ admits presentation with 
$\{(ij) \mid i\neq j,\ 1\leq i,j\leq n\}$ as the set of generators and the following list of relations
 (in every formula distinct letters denote distinct indices):
\begin{align}
(ij)^2 = &\, 1,         \label{Sym1} \\
(ij)^{(jk)} = &\, (ik), \label{Sym2} \\
[(ij), (kl)] = &\,1.    \label{Sym3} \\
(ij) =&\, (ji),         \label{Sym0}
\end{align}
\end{lemma}
\begin{rem}
To simplify notation, up to the end of this section we adhere to the convention that in our formulae distinct letters denote distinct indices.
\end{rem}

Now we are ready to formulate the main result of this subsection.
\begin{prop} \label{prop:Q-pres} For $n\geq 3$ and arbitrary group $G$ the group $S_n(G)$ admits presentation with 
$\{(ij)_a \mid i\neq j, 1\leq i,j,\leq n, a\in G\}$ 
as the set of generators and the following list of relations.
\begin{align}
(ij)_a^2 = &\, 1,                \label{Q1} \tag{SG1} \\
(ij)_a^{(jk)_b} = &\, (ik)_{ab}, \label{Q2} \tag{SG2} \\
[(ij)_a, (kl)_b] = &\,1.         \label{Q3} \tag{SG3} \\
(ij)_a =&\, (ji)_{a^{-1}}        \label{Q4} \tag{SG4} 
\end{align}
\end{prop}
\begin{proof}
Denote by $S'$ the group from the statement of the proposition.
For $1\leq i\leq n$ and $x\in G$ denote by $x[i]$ the element of $G^n$ 
 whose only nontrivial component equals $x$ and is located in the $i$-th position.

It is not hard to deduce from the definition of $S_n(G)$ and~\cref{lem:transp-deff} 
 that the formula $(ij)_a \mapsto (ij)_{a[j]}$ gives a well-defined map $\varphi\colon S'\to S_n(G)$.

Now we are going to construct the map $\psi\colon S_n(G)\to S'$ in the opposite direction. 
Using the presentation of $S_n$ given by~\cref{lm:Snpres}  we define for a fixed $g\in G^n$ the map 
 $\psi_g\colon S_n \to S'$ by $\psi_g((ij)) = (ij)_{g_i^{-1} g_j}.$
It is obvious that $\psi_g$ preserves the defining relations \eqref{Sym1}--\eqref{Sym0} of $S_n$. 
It remains to show that the equality $\psi_g(s) = \psi_h(s)$ holds whenever $g, h\in G^n$ and $s\in S_n$ are such that $s_g=s_h$.

Indeed, if $hg^{-1}$ is fixed by $s$ then for every $1\leq i\leq n$ we have $(hg^{-1})_i = (hg^{-1})_{s(i)}$, or equivalently
$h_i^{-1} h_{s(i)} = g_i^{-1} g_{s(i)}$. 
For example, if $s$ is a cycle of length $p$, i.\,e. $s=(i_1, i_2, \ldots i_p)$ with $i_{k+1} = s(i_k)$ we get that
\begin{multline} \nonumber
 \psi_g(s) = \psi_g\left({\prod\limits_{k=1}^{p-1}(i_k, i_{k+1})}\right) = 
 \prod\limits_{k=1}^{p-1}\left(i_k, i_{k+1}\right)_{g_{i_k}^{-1} g_{i_{k+1}}} = 
 \prod\limits_{k=1}^{p-1}\left(i_k, i_{k+1}\right)_{h_{i_k}^{-1} h_{i_{k+1}}} = \psi_h(s).
\end{multline}
The proof in the general case is similar.
Verification of the fact that $\psi$ and $\varphi$ are mutually inverse is also immediate.
\end{proof}

\subsection{Coxeter-like presentation for $S_n(G)$.}
The aim of this subsection is to show the following.
\begin{prop}\label{prop:Cox-Amalgam}
The group $S_n(G)$ defined in~\cref{prop:Q-pres} coincides with the group defined in the introduction by means of Coxeter-like presentation
(denote the latter group by $S_n^{C}(G)$).
\end{prop}

The idea of the proof is to construct a sequence of intermediate groups
\[S_n^C(G) \to S_n(1, G) \to \ldots \to S_n(n-1,G) \to S_n(G)\]
and then show that all the maps between these groups are isomorphisms.

\begin{df}
For $1\leq t\leq n-1$ consider the group $S_n(t, G)$ defined by the set of generators 
$\{(ij)_a \mid 1\leq i<j \leq n,\ j-i\leq t\}$ subject to the following relations:
\begin{align}
(ij)_a^2                      = &\, 1,                   &               \label{CoxMix1} \\
(ij)_a^{(jk)_b}               = &\, (jk)_{b'}^{(ij)_{a'}}, & \text{provided } ab = a'b',      \label{CoxMix2} \\
(ij)_a (kl)_b                 = &\, (kl)_b (ij)_a,       &               \label{CoxMix3} \\
(ij)_a ^ {(jk)_b}             = &\, (ik)_{ab}.           &               \label{CoxMix4}
\end{align}
In the above relations indices $i,j,k,l$ satisfy the inequalities ensuring that all the generators appearing in the formulae are well-defined.
\end{df}

In the case $t=1$ the relation~\eqref{CoxMix4} becomes vacuous and the remaining relations have the same form as \eqref{Cox1}---\eqref{Cox3}.
Thus, the map $S_n^C(G) \to S_n(1, G)$ defined by $s_i(g)\mapsto (i,i+1)_g$ is an isomorphism. 

On the other hand, in the case $t=n-1$ the obvious embedding of generators induces a map $S_n(n-1,G) \to S_n(G)$ 
 (we consider the latter group in the presentation of~\cref{prop:Q-pres}).
We claim that this map is an isomorphism. Indeed, define the map $\theta\colon S_n(G) \to S_n(n-1, G)$ in the opposite direction by the following formula:
\[\theta((ij)_a) = \left\{\def\arraystretch{1.2}%
  \begin{array}{@{}c@{\quad}l@{}}
    (ij)_a & \text{if $i < j$,}\\
    (ji)_{a^{-1}} & \text{if $j < i$.}\\
  \end{array}\right.\]
Only the fact that $\theta$ preserves the relation~\eqref{Q2} is not immediately obvious.
Fix three indices $i,j,k$ satisfying $i<j<k$ and let $p,q,r$ be any permutation of these indices.
One has to check that the equality $\theta((pq)_a^{(qr)_b}) = \theta((pr)_{ab})$ holds.
For example, the relation $\theta((jk)_a^{(ki)_b}) = \theta((ji)_{ab})$ can be proved as follows:
\[ \theta((jk)_a^{(ki)_b}) = (jk)_a^{(ik)_{b^{-1}}} = {(ik)_{b^{-1}}}^{(jk)_a} = (ij)_{b^{-1}a^{-1}} = \theta((ji)_{ab}).\]
In the above formula the second equality holds by~\eqref{CoxMix2} and the third one holds by~\eqref{CoxMix4} (in which both sides are conjugated by $(jk)_a$).
Verification of the other 5 equalities is similar and is left as an exercise.

To finish the proof of~\cref{prop:Cox-Amalgam} it remains to show the following lemma.
\begin{lemma} \label{lem:coxeter-amalgam}
For each $1 \leq t \leq n-2$ the map $f_t \colon S_n(t, G)\to S_n(t+1, G)$ induced by the obvious embedding of generators is an isomorphism.
\end{lemma}
\begin{proof}
Let $(ij)_a$ be a generator of $S_n(t+1, G)$, define the map $g_t$ inverse to $f_t$ as follows:
\[g_t((ij)_a) = \left\{\def\arraystretch{1.2}%
  \begin{array}{@{}c@{\quad}l@{}}
    (ij)_a & \text{if $j-i\leq t$,}\\
    (ik)_{a}^{(kj)_1} & \text{if $j-i=t+1$ for some $i<k<j$.}\\
  \end{array}\right.\]
The value of $g_t((ij)_a)$ does not depend on the choice of $k$.
Indeed, let $k'> k$ be another index lying between $i$ and $j$. 
By~\eqref{CoxMix3} we have the equality $(ik)_a^{(k'j)_1} = (ik)_a$ hence
\[(ik)_a ^{(kj)_1} = (ik)_a ^{\left((kk')_1 ^{(k'j)_1}\right)} = (ik)_a^{(kk')_{1} (k'j)_{1}} = (ik')_a ^{(k'j)_1}.\]

Let us show that $g_t$ is well defined (i.\,e. preserves the defining relations of $S_n(t+1, G)$).
Below we outline the proof of this fact for relation~\eqref{CoxMix2} and leave the verification for other relations as an exercise.

First consider the case $k-j\leq t$ and $j-i=t+1$. Choose some index $j'$ between $i$ and $j$.
Using the defining relations of $S_n(t,G)$ we get the following chain of equalities:
\begin{multline}\label{CoxMix2New} g_t((ij)_a^{(jk)_b}) = (ij')_{a} ^ {(j'j)_1 (jk)_b} = (ij')_a ^{ (jk)_b (j'j)_1 (jk)_b } = 
(ij')_a ^{(jk)_1 (j'j)_b (jk)_1} = \\ = (ij')_a ^{(j'j)_b (jk)_1} = (j'j)_{b'}^ {(ij')_{a'} (jk)_1} = 
 (j'j)_{b'}^ {(jk)_1 (ij')_{a'}} = \\ = (jk)_{b'}^{(j'j)_1 (ij')_{a'}} = (jk)_{b'}^{(ij')_{a'} (j'j)_1 (ij')_{a'}}
 = (jk)_{b'}^{(j'j)_1 (ij')_{a'} (j'j)_1} = g_t((jk)_{b'}^{(ij)_{a'}}). \end{multline}
The proof in the case $k-j = t+1$ is similar:
\[g_t((ij)_a^{(jk)_b}) = (ij)_a ^ {(k'k)_1 (jk')_b (k'k)_1} = (ij)_a ^ {(jk')_b (k'k)_1} =
(jk')_{b'} ^ {(ij)_{a'} (k'k)_1} = (jk')_{b'} ^ {(k'k)_1 (ij)_{a'}} = g_t((jk)_{b'}^{(ij)_{a'}}).\]
Here $k'$ is some index lying between $j$ and $k$.
In the central equality we use either~\eqref{CoxMix2} or~\eqref{CoxMix2New} depending on whether $j-i\leq t$ or $j-i=t+1$.
\end{proof} 

\section{Comparison with extensions of type \texorpdfstring{$\mathfrak{H}_n(G)$}{Hn(G)}} \label{sec:Hnextensions}
The aim of this section is to express the group $S_n(G)$ in terms of Rehmann's extensions of
 type $\mathfrak{H}_n$ and thus finish the proof of~\cref{thm:summary}. 

\subsection{Presentation for the subgroup $HS_n(G)$}
Our first goal is to obtain an explicit presentation of the subgroup $HS_n(G)$ of $S_n(G)$.

For every $i\neq j$ and $a\in G$ we define the element $h_{ij}(a) \in HS_n(G)$ as follows:
\begin{equation} \label{eq:h-def} 
h_{ij}(a) = (ij)_{a} \cdot (ij)_1. 
\end{equation}
It is not hard to show that $h_{ij}(a)$ form a generating set for $HS_n(G)$.
In fact, there is an explicit formula how an element of $HS_n(G)$ originally expressed through $(ij)_{a}$'s can be rewritten in terms of $h_{ij}(a)$.
Indeed, if $h$ lies in $HS_n(G)$ and is written as $\prod_{k=1}^N(i_k j_k)_{a_k}$ for some $i_k\neq j_k$ and $a_k\in G$ then it can be rewritten as follows:
\begin{equation} \label{eq:rp} \tag{$\tau$}
 h = \prod_{k=1}^N h_{\sigma_k(i_k), \sigma_k(j_k)}(a_k),\text{ where } \sigma_k=\prod_{s=1}^{k-1} (i_s j_s) \in S_n. \end{equation}

We briefly recall the notion of a {\it rewriting process} given in~\cite[\S~2.3]{MKS76}.
If $G$ is a group presented by generators $a_\nu$ and relations $R_{\mu}(a_{\nu})$ and $H$ is its subgroup with a generating set $J_i(a_\nu)$ then
 a {\it rewriting process for $H$} is a function which maps every word $u$ in alphabet $a_\nu$ to a word $v$ in alphabet $s_i$ such that
 $u$ and $v[s_i:=J_i]$ define the same element of $G$ whenever $u$ represents an element of $H$.

With this terminology, the mapping~\eqref{eq:rp} defined above is a rewriting process for the subgroup $HS_n(G)$.
Since it does not arise as a rewriting process corresponding to a coset representative function,
this process is {\it not} a Reidemeister rewriting process in the sense of~\cite[\S~2.3]{MKS76}.
However, it still satisfies the following two key properties of a Reidemeister rewriting process (cf. with (v) and (vi) of~\cite[\S~2.3]{MKS76}): 
\begin{itemize}
 \item if $U$ and $U^*$ are freely equal words in $(ij)_a$ then $\tau(U)$ and $\tau(U^*)$ are also freely equal words in $h_{ij}(a)$;
 \item if $U_1$ and $U_2$ are two words in $(ij)_a$ defining certain elements of $HS_n(G)$ then the words $\tau(U_1U_2)$ and $\tau(U_1) \tau(U_2)$ are equal.
\end{itemize}
Using these two properties and repeating the arguments used in the proof of~\cite[Theorem~2.8]{MKS76}
one can simplify the generic presentation of $HS_n(G)$ from~\cite[Theorem~2.6]{MKS76} and obtain the following.
\begin{lemma} \label{lm:h-gen}

 For $n\geq 3$ the group $HS_n(G)$ admits presentation on the generators $h_{ij}(a)$ with the following two families of defining relations:
 \begin{align}
  h_{ij}(a) = &\, \tau\left((ij)_a \cdot (ij)_1\right); \label{eq:tau1} \\
  \tau(KRK^{-1}) = &\, 1, \label{eq:tau2} 
 \end{align}  
 where $R$ varies over relations of~\cref{prop:Q-pres} and $K$ is any word in $(ij)_1$, $i\neq j$.
\end{lemma}

From~\eqref{eq:tau1} one immediately obtains the equality $h_{ij}(1)=1$, $i\neq j$.
Thus, if we denote by $\sigma$ the permutation corresponding to a word $K$ and let $S_n$ act on $h_{ij}(a)$'s in the natural way
 we get that the word $\tau(KRK^{-1})$ is equivalent to ${}^{\sigma}\tau(R)$ (modulo relations $h_{ij}(1)=1$). 

Since the image of every relator from~\cref{prop:Q-pres} under the action of $S_n$
 is still a relator of the same form, only relations of the form $\tau(R)=1$ are, in fact, 
 needed for the presentation of $HS_n(G)$.
Writing down what $\tau(R)$ is for each of \eqref{Q1}--\eqref{Q4} we get the following.
 
\begin{prop} \label{prop:HSpres} For $n\geq 3$ the group $HS_n(G)$ admits presentation with generators $h_{ij}(a)$ and the following list of relations:
\begin{align}
h_{ij}(1)                     = &\, 1,              \tag{R0} \label{H0} \\
h_{ij}(a) h_{ji}(a)           = &\, 1,              \tag{R1} \label{H1} \\
h_{jk}(b) h_{ik}(a) h_{ij}(b) = &\, h_{ik}(ab),     \tag{R2} \label{H2} \\
[h_{ij}(a), h_{kl}(b)]        = &\, 1,              \tag{R3} \label{H3} \\
h_{ij}(a)^{-1}                = &\, h_{ij}(a^{-1}). \tag{R4} \label{H4}
\end{align}
\end{prop}

In the next subsection we put the group $HS_n(G)$ into the context of Rehmann's extensions of type $\mathfrak{H}_n(G)$.
\subsection{Extensions of type $\mathfrak{H}_n$}
We start with a brief review of the material of \S~1--3 of~\cite{Reh78}.
For $n\geq 3$ denote by $D_n(G)$ the subgroup of $G^n$ consisting of vectors
$(g_1,\ldots, g_n)$ for which the product $g_1 g_2 \ldots g_n$ lies in the derived subgroup $[G, G]$.
Clearly, $D_n(G)$ is a normal subgroup of $G^n$ and the quotient $G^n/D_n(G)$ is isomorphic to $G^{\mathrm{ab}}\cong \HH_1(G, \ZZ)$.
The subgroup $D_n(G)$ is generated by elements $d_{ij}(g)$, where $d_{ij}(g)$ denotes the element of $G^n$ whose
 $i$-th component equals $g$, $j$-th component equals $g^{-1}$ and all other components are trivial.

The following definition is copied from \cite[\S~2]{Reh78}.
Consider the group $H_n(G)$ defined by generators $h_{ij}(u)$, $u\in G$, $i\neq j$ subject to the following relations:
\begin{align}
h_{ij}(u) h_{ji}(u)                = &\, 1,                        &                     \tag{H1} \label{RH1} \\
h_{ij}(u) h_{ki}(u) h_{jk}(u)      = &\, 1,                        &                     \tag{H2} \label{RH2} \\
h_{ij}(u) h_{ik}(v) h_{ij}(u)^{-1} = &\, h_{ik}(uv) h_{ik}(u)^{-1},& \text{for } j\neq k \tag{H3} \label{RH3} \\
h_{ij}(u) h_{kj}(v) h_{ij}(u)^{-1} = &\, h_{kj}(vu) h_{kj}(u)^{-1},& \text{for } i\neq k \tag{H4} \label{RH4} \\
[h_{ij}(u), h_{kl}(v)]             = &\, 1.                        &                     \tag{H5} \label{RH5} \end{align}
There is a surjective group homomorphism $H_n(G) \twoheadrightarrow D_n(G)$ sending $h_{ij}(u)$ to $d_{ij}(u)$.

By definition, {\it an extension of type $\mathfrak{H}_n(G)$} is an extension $H$ of $D_n(G)$ which is also a quotient of $H_n(G)$, i.\,e. 
 the one which fits into the following diagram.
\[ \xymatrix{ H_n(G) \ar@{->>}[rd] \ar@{->>}[d] & \\ H \ar@{->>}[r] & D_n(G)} \]

For $1\leq k\leq n$ denote by $\iota_k$ the map which sends an element $g$ of $[G, G]$ to $g[k]$ (the vector of $G^n$
 whose only nontrivial component equals $g$ and is located on $k$-th position).
Denote by $U_{H_n(G)}$ the pull-back in the following diagram:
\[ \xymatrix{ U_{H_n(G)} \pullbackcorner \ar@{^{(}->}[d] \ar@{->>}[r] & [G, G] \ar@{^{(}->}[d]^{\iota_k} \\ H_n(G) \ar@{->>}[r] & D_n(G)}\]
It turns out, that the group $U_{H_n(G)}$ is isomorphic to $G \mathbin{\widetilde{\otimes}} G$ for any $k$.
Here $G \mathbin{\widetilde{\otimes}} G$  denotes the nonabelian antisymmetric square of $G$ 
  (it is denoted by $U(G)$ in~\cite{Reh78} and by $(G, G)$ in~\cite{De76}).

A quotient of $G \mathbin{\widetilde{\otimes}} G$ is called an {\it extension of type $\mathfrak{U}(G)$} (cf.~\cite[\S~1]{Reh78}).
There is one-to-one correspondence between extensions of type $\mathfrak{H}_n(G)$ and those of type $\mathfrak{U}(G)$.
One obtains a $\mathfrak{U}(G)$-extension $U_H$ from a given $\mathfrak{H}_n(G)$-extension $H$ by taking pull-back as in the above diagram.
In other words, $U_H$ is the subgroup of $H$ generated by all the symbols $c_{kj}(u, v)$, where
\[c_{kj}(u,v)=h_{kj}(u)h_{kj}(v)h_{kj}(vu)^{-1}, \ \ 1\leq k \neq j \leq n,\ u,v \in G.\]
(It can be shown that symbols $c_{kj}(u, v)$ do not depend on $j$).
The inverse construction which produces an $\mathfrak{H}_n(G)$ extension from the given $\mathfrak{U}(G)$-extension and a number $n$ is described in detail in~\cite[\S~3]{Reh78}.

We now proceed with the proof of~\cref{thm:summary}.
Our main result is the following.
\begin{prop} \label{prop:simpler} 
For $n\geq 3$ relations \eqref{H0}--\eqref{H4} imply \eqref{RH1}--\eqref{RH5}.
Moreover, for $n\geq 4$ relations \eqref{H0}--\eqref{H3} are equivalent to \eqref{RH1}--\eqref{RH5}.
\end{prop}
For $u, v\in G$ define the following symbol:
\[ c'_{ij}(u,v)=h_{ij}(u)h_{ij}(v)h_{ij}(uv)^{-1}.\]
Notice that~\eqref{RH4} implies $h_{ij}(1)=1$ therefore $c_{ij}(u, u^{-1}) = c'_{ij}(u, u^{-1})$.
Let us show that~\eqref{RH3} can be omitted from the definition of $H_n(G)$ provided $n \geq 4$.
\begin{lemma} \label{item-lem33} If one excludes relation~\eqref{RH3} from the list of defining relations for the group $H_n(G)$, $n\geq 3$,
 the following assertions still remain true:
 \begin{lemlist}
\item \label{item-lem33-cntr} The elements $c_{ij}(u, u^{-1})$ are central in $H_n(G)$;
\item \label{item-lem33-comm} One has $c'_{ij}(u, v) = [h_{ij}(u), h_{kj}(v)]$, $k\neq i, j$;
\item \label{item-lem33-conj}  One has ${}^{h_{ij}(w)}c_{kj}'(u, v) = c_{kj}'(u, w)^{-1} c_{kj}'(u, vw)$, $k\neq i, j$;
\item \label{item-lem33-conj2} One has ${}^{h_{ij}(w)} c'_{ij}(u, v) = c'_{ij}(uw, v) c'_{ij}(w, v)^{-1}$.
 \end{lemlist}
\end{lemma}
\begin{proof}
First one shows using~\eqref{RH4} that $c_{ij}(u, u^{-1})$ centralizes $h_{kj}(v)$ (cf. with the proof of~\cite[Lemma~2.1(2)]{Reh78}).
Since in any group $[a, b]=1$ implies $[a^{-1}, b] = [a, b^{-1}] = [a^{-1}, b^{-1}] = 1$ we get that $c_{ij}(u, u^{-1})$ also centralizes $h_{kj}(v)^{-1} = h_{jk}(v)$ and
 $c_{ij}(u^{-1}, u)^{-1} = c_{ji}(u, u^{-1})$ centralizes both $h_{kj}(v)$ and $h_{jk}(v)$.
Together with~\eqref{RH2} and~\eqref{RH5} this implies that $c_{ij}(u, u^{-1})$ centralizes all the generators of $H_n(G)$ and hence lies in the center of $H_n(G)$.

The second and the third assertions are straightforward corollaries of~\eqref{RH4}.
The fourth assertion follows from the third one:
\[{}^{h_{ij}(w)} c'_{ij}(u, v) = {}^{h_{ij}(w)} c'_{kj}(v, u)^{-1} = (c'_{kj}(v, w)^{-1} c'_{kj}(v, uw))^{-1} = c'_{ij}(uw, v) c'_{ij}(w, v)^{-1}. \qedhere\]
\end{proof}

\begin{lemma}
If one excludes relation~\eqref{RH3} from the list of defining relations for the group $H_n(G)$, $n\geq 3$, the following statements are equivalent:
\begin{enumerate}
\item \label{item1} \eqref{RH3} holds;
\item \label{item2} one has $c'_{ij}(u, v)^{-1} = c'_{ij}(v, u)$;
\item \label{item3} symbols $c'_{ij}(u, v)$ do not depend on $i$;
\item \label{item4} one has $c'_{ij}(u, vw) = c'_{ij}(wu, v) \cdot c'_{ij}(uv, w)$;
\item \label{item5} one has $c'_{ij}(u, vu^{-1}) = c'_{ij}(uv, u^{-1})$.
\end{enumerate}
\end{lemma}
\begin{proof}
Implications $(1) \implies (2) \implies (3) \implies (4)$ can proved by essentially the same argument as 
 the corresponding statements for $c_{ij}(u, v)$ contained in~\cite[Lemmas~2.1-2.2]{Reh78}.
Implication $(4) \implies (5)$ is trivial.

We now prove $(5) \implies (1)$. Notice that~\eqref{RH4} implies ${}^{h_{ij}(u)^{-1}}h_{ik}(v) = h_{ik}(u)^{-1} h_{ik}(vu)$, therefore
using~\cref{item-lem33-cntr} we get that ${}^{h_{ij}(u)}h_{ik}(v) = {}^{c'_{ij}(u, u^{-1}) h_{ij}(u^{-1})^{-1}}h_{ik}(v) = h_{ik}(u^{-1})^{-1} h_{ik}(vu^{-1}).$
Thus, \eqref{RH3} is equivalent to the equality 
\[ h_{ik}(uv) h_{ik}(u)^{-1} = h_{ik}(u^{-1})^{-1} h_{ik}(vu^{-1}),\]
or what is the same
\[c'_{ik}(u^{-1}, uv) = h_{ik}(u^{-1}) h_{ik}(uv) =  h_{ik}(vu^{-1}) h_{ik}(u) = c'_{ik}(vu^{-1}, u).\]
It is clear that the last equality is an equivalent reformulation of $(5)$.
\end{proof}

\begin{proof}[Proof of~\cref{prop:simpler}]
Notice first that for $n\geq 4$ the third statement of the previous lemma is automatically true.
Indeed, from~\cref{item-lem33-comm} it follows that
\[ c'_{ij}(u, v) = c'_{sj}(v, u)^{-1} = c'_{kj}(u, v), \text{ where } s\neq i,j,k. \]
Thus, \eqref{RH3} is superfluous for $n\geq 4$.
To prove the proposition it remains to notice the following:
\begin{itemize}
 \item \eqref{RH2} and \eqref{RH4} follow from \eqref{H0}--\eqref{H2};
 \item \eqref{H0} follows from \eqref{RH4};
 \item \eqref{H2} follows from \eqref{RH1}, \eqref{RH2} and \eqref{RH4};
 \item \eqref{RH4} and \eqref{RH3} are equivalent in the presence of \eqref{H4}; 
\end{itemize}
\end{proof}

\begin{cor} \label{cor:main} For $n \geq 3$ the group $HS_n(G)$ is the quotient of the group $H_n(G)$ by~\eqref{H4} (in particular, 
 it is an extension of type $\mathfrak{H}_n(G)$).
Moreover, there is an isomorphism, which is natural in $G$:
\begin{equation} \Ker(S_n(G) \to G \wr S_n) = \Ker(HS_n(G) \to D_n(G)) \cong \HH_2(G, \ZZ). \end{equation} \end{cor}
\begin{proof}
Notice that imposing~\eqref{H4} has the same effect as modding out by all symbols $c_{ij}(u, u) = c_{ij}(u, u^{-1})$ (cf.~\cite[p.~87]{Reh78}).
The corollary and \cref{thm:summary} now follow from the above proposition and~\cite[Proposition~5]{De76}.
\end{proof}
Notice that the extension of type $\mathfrak{U}(G)$ corresponding to $HS_n(G)$ is isomorphic to the exterior square $G\wedge G$. 

\section{Proof of Theorem~2} \label{sec:main}
Recall that for arbitrary set $X$ one defines the space $EX$ as the the simplicial set whose set of $k$-simplices $EX_k$
 is $X^{k+1}$ and whose faces and degeneracies are obtained by omitting and repeating components. 
For a group $G$ we denote by $\pi_G$ the canonical map $EG \to BG$ sending $(g, h) \in EG_1$ to $g^{-1}h \in BG_1$.
 
Now let $N$ be a group acting on $X$. We define two simplicial sets $U$ and $V$ as follows:
\[ U = \bigcup\limits_{n\in N} E(\Gamma_{n.-}) \subseteq E(X\times X),\ \ V = \bigcup\limits_{x,y\in X}E(N(x\to y)) \subseteq EN. \]
Here $\Gamma_{n.-}$ is the graph of the function $(x \mapsto nx)$ and $N(x\to y)$ denotes the subset of elements $n\in N$ satisfying $nx=y$.
With this notation the subset $N(x\to x)$ coincides with the stabilizer subgroup $N_x \leq N$.

\begin{lemma} \label{lm:quillen-a} The simplicial sets $U$ and $V$ are homotopy equivalent. \end{lemma}
\begin{proof} First, we define yet another simplicial set $W$ as follows.
Its $k$-simplices $W_k$ are matrices $\left(\begin{smallmatrix}x_0 & x_1 & \ldots & x_k&\\ n_0 & n_1 & \ldots & n_k \end{smallmatrix}\right)$,
 where $x_i\in X$ are $n_i\in N$ are such that all $n_i$'s act on each $x_j$ in the same way, i.\,e. $n_ix_j = n_{i'} x_j$ for $0\leq i,i',j\leq k$. 
 The faces and degeneracies of $W$ are the maps of omission and repetition of columns.
 
 Now there are two simplicial maps $f\colon W\to U$, $g\colon W\to V$ whose action on $0$-simplicies is given by 
  $f\left(\begin{smallmatrix}x_0 \\ n_0\end{smallmatrix}\right) = (x_0, n_0x_0)$, 
  $g\left(\begin{smallmatrix}x_0 \\ n_0\end{smallmatrix}\right) = n_0$. 
 To prove the lemma it suffices to show that $f$ and $g$ are homotopy equivalences. 
 The proof for $f$ and $g$ is similar, let us show, for example, that $g$ is a homotopy equivalence.
 
 In view of Quillen theorem A (cf.\cite[ex.~IV.3.11]{Kbook}) it suffices to show that for each $p$-simplex $d \colon \Delta^p \to V$ the 
  pullback $g/(p, d)$ of $d$ and $g$ is contractible.
 The simplicial set $g/(p, d)$ can be interpreted as the subset of $\Delta^p \times EX$ whose set of $k$-simplices consists of pairs
  $(\alpha\colon \underline{k}\to \underline{p}, (x_0, \ldots, x_k)\in X(\alpha, d)^{k+1})$.
 Here $X(\alpha, d)$ is the subset of $X$ consisting of all $x$ for which $d_{\alpha(i)}x = d_{\alpha(j)}x$ for $0\leq i,j\leq k$.
 Notice that the set $X_d := X(id_{\underline{p}}, d)$ is nonempty and is contained in every $X(\alpha, d)$ (it even equals $X(\alpha, d)$ for surjective $\alpha$).
 Now choose a point $\widetilde{x}\in X_d$ and consider the simplicial homotopy \[H\colon \Delta^p \times EX \times \Delta^1 \to \Delta^p\times EX\] 
  between the identity map of $\Delta^p \times EX$ and
 the map $\Delta^p \times c_{\widetilde{x}}$, where $c_{\widetilde{x}}$ is the constant map. 
 More concretely, $H$ sends each triple $(\alpha\colon \underline{k} \to \underline{p}, (x_0, \ldots, x_k), \beta\colon \underline{k}\to\underline{1})$
 to $(\alpha, (x_0, \ldots, x_{i-1}, \widetilde{x}, \ldots, \widetilde{x}))$, where $i$ is the minimal number such that $\beta(i)=1$.
 By the choice of $\widetilde{x}$ the image of $H$ restricted to $g/(p, d)\times \Delta^1$ is contained in $g/(p, d)$, hence $g/(p, d)$ is contractible. 
\end{proof}

Now suppose that $X=H$ is also a group upon which $N$ acts on the left.
\begin{cor} \label{cor:ker-iso}
Consider the following two simplicial sets:
\[ S = \bigcup\limits_{h\in H} BN_h \subseteq BN,\ \ \ T = \pi_N(U) \subseteq B(H \times H).\]
There is an isomorphism $\theta\colon \Ker(\pi_1(S) \to N) \cong \Ker(\pi_1(T) \to H \times H).$
Moreover, the higher homotopy groups of $S$ and $T$ are isomorphic. 
\end{cor}
\begin{proof}
Consider the following two pull-back squares:
\[ \xymatrix{ V  \ar@{^{(}->}[r] \ar[d] \pullbackcorner & \ar[d]^{\pi_N} EN \\
              S \ar@{^{(}->}[r] & BN } \ \ \ 
   \xymatrix{ U  \ar@{^{(}->}[r] \ar[d] \pullbackcorner & \ar[d]^{\pi_{H \times H}} E(H \times H) \\
              T \ar@{^{(}->}[r] & B(H \times H)}\] 
The required isomorphism can be obtained from the homotopy long exact sequence applied to the left arrows of these diagrams.\end{proof}

Now let $G$ be a group. Set $N = G \wr S_n$, $H = G^n$ and consider the left action of $N$ on $H$ given by $(g, s) \cdot h = gh^{s^{-1}}$, $g, h\in G^n$, $s\in S_n$.
The simplicial subset $T \subset BG^{2n}$ from the above corollary is precisely the preimage of $0$ under
 the map $BG^{2n} \to \ZZ[BG]$ defined by \[(x_1, x_2, \ldots x_{2n}) \mapsto x_1 + \ldots + x_n - x_{n+1} - \ldots - x_{2n}.\] 
The latter map coincides with the alternating map $h_n$ from the introduction up to a permutation of factors of $BG^{2n}$,
 therefore we can identify $T$ with $h_n^{-1}(0)$.

It is also easy to compute the map $\pi_1(S) \to N$. Indeed, van Kampen theorem~\cite[Theorem~2.7]{May99} asserts that
$\pi_1(S)$ is isomorphic to the free product of stabilizer subgroups $N_{h} \leq N$ amalgamated over pairwise intersections $N_h \cap N_{h'}$, $h, h\in H$.
For $h \in G^n$ the subgroup $N_h$ consists of elements $(g, s) \in N$ satisfying $gh^{s^{-1}} = h$, i.e. elements of the form $(hh^{-s^{-1}}, s)$.
Thus, there is an isomorphism $N_h\cong S_n$, the fundamental group $\pi_1(S)$ is isomorphic to $S_n(G)$ and the map $\pi_1(S) \to N$ coincides with the map $\mu$ defined in section~\ref{sec:QnG-def}.
In particular, for an abelian $G$ its kernel is generated by symbols $c_{ij}(u, v)$.

We have shown that $\Ker(\mu)$ is isomorphic to $\Ker(\pi_1(h_n(0)) \to G^{2n})$.
This isomorphism is natural in $G$, since all the objects $N$, $H\times H$, $S$, $T$ used in its construction depended on $G$ functorially.
The following lemma gives even more concrete description of this isomorphism for abelian $G$.
\begin{lemma} \label{lem:concrete-formula}
 If $G$ is abelian then every generator $c_{ij}(u, v) \in \Ker(\pi_1(S)\to N)$ is mapped under $\theta$ to the class (in $\pi_1(T)$) of the loop $f_{ij}(u, v)$, 
  which is defined as follows:
 \[f_{ij}(u,v) = p_{ij}(u) \circ p_{ij}(v) \circ p_{ij}(u^{-1} v^{-1}), \text{ where }
   p_{ij}(x)=(x[i], x[i]) \circ (x^{-1}[i], x^{-1}[j]).\]
 Here we denote by $x \circ y$ the formal concatenation of two loops from $T_1$.
\end{lemma}
\begin{proof}
Fix a vector $g\in G^n$ and denote by $\gamma_{ij}(g, x)$ the path of length $4$ in $W$ (the simplicial set defined in the proof of~\cref{lm:quillen-a}) which connects 
the following five points of $W_0$:
\[
\Big(\begin{smallmatrix} e    \\ (g; e) \end{smallmatrix}\Big),\,
\Big(\begin{smallmatrix} x[i] \\ (g; e) \end{smallmatrix}\Big),\,
\Big(\begin{smallmatrix} x[i] \\ (d_{ij}(x) g, (ij)) \end{smallmatrix}\Big),\,
\Big(\begin{smallmatrix} e    \\ (d_{ij}(x) g, (ij)) \end{smallmatrix}\Big),\,
\Big(\begin{smallmatrix} e    \\ (d_{ij}(x) g, e) \end{smallmatrix}\Big).\]
One can easily check that the image of $\gamma_{ij}(g, x)$ in $S$ under $\pi_N g$ is precisely the element $h_{ij}(x)$, while
$\pi_{H\times H} f$ sends $\gamma_{ij}(g, x)$ to the loop $p_{ij}(x)$ ($f, g$ are also defined in the proof of~\cref{lm:quillen-a}).
\end{proof}

\begin{lemma} \label{lm:endotr} The only natural endotransformations of the functor $\HH_2(-, \ZZ)\colon \catname{Groups}\to \catname{Ab}$ 
 are morphisms of multiplication by $n \in \ZZ$.
\end{lemma}
\begin{proof}
 Denote by $\eta$ an endotransformation $\HH_2(-, \ZZ) \to \HH_2(-, \ZZ)$.
 When restricted to the subcategory of free finitely-generated abelian groups $\catname{Add}(\ZZ) \subseteq \catname{Ab}$ the second homology functor
 coincides with the second exterior power functor $A \mapsto \wedge^2A$. 
 
 Recall from~\cite[Theorem~6.13.12]{Ba96} that the category of quadratic functors is equivalent to the category of quadratic $\ZZ$-modules (see Definition~6.13.5 ibid.)
 The functor $A \mapsto \wedge^2A$ is clearly quadratic and corresponds to the quadratic $\ZZ$-module $0 \to \ZZ \to 0$ under this equivalence.
 Thus, we get that $\eta$ restricted to $\catname{Add}(\ZZ)$ coincides with the the morphism of multiplication by $n\in \ZZ$.
 
 Consider the group $\Gamma_k = \langle x_1, y_1, \ldots x_k, y_k \mid [x_1, y_1]\cdot \ldots \cdot [x_k, y_k] \rangle$ (the fundamental group of a sphere with $k$ handles).
 It is clear that the abelianization map $\Gamma_k \to \ZZ^{2k}$ induces an injection $\HH_2(\Gamma_k, \ZZ) \cong \ZZ \to \wedge^2\ZZ^{2k}$.  
 Consider the following diagrams.
  \[ \xymatrix{ \ZZ \ar@{^{(}->}[r] \ar[d]_{\eta_{\Gamma_k}} & \ar[d]^{n \cdot} \wedge^2\ZZ^{2k} \\
                \ZZ \ar@{^{(}->}[r]                          & \wedge^2\ZZ^{2k} } \ \ \ 
     \xymatrix{   \ZZ  \ar[r]^(.35){\chi} \ar[d]_{\eta_{\Gamma_k}} & \ar[d]^{\eta_G} \HH_2(G, \ZZ) \\
                  \ZZ  \ar[r]^(.35){\chi}                          & \HH_2(G, \ZZ)}  \]
 From the left diagram it follows that $\eta_{\Gamma_k}$ is also the morphism of multiplication by $n$.
 For every element $x\in \HH_2(G, \ZZ)$ there exist an integer $k$ and a map $\chi\colon \Gamma_k\to G$ sending the generator of $\HH_2(\Gamma_k, \ZZ)$ to $x$.
 From the right square we conclude that $\eta_G(x) = nx$, as claimed.
\end{proof}

\begin{proof}[Proof of~\cref{thm:main}]
Factor $h_n$ as the composition of a trivial cofibration and a fibration \[BG^{2n}\to E_{h_n} \to \ZZ[BG].\]
Now write down the starting portion of the long homotopy exact sequence for the latter.
If we denote by $K$ the kernel of the map $\nu\colon \pi_1(T) \to G^{2n}$ induced by the embedding $h_n^{-1}(0)=T \subseteq BG^{2n}$ we obtain the following commutative diagram:
\[ \xymatrix{1 \ar[r] & K       \ar[d]_{\psi} \ar[r]_{\beta} & \pi_1(T) \ar[d]_{\varphi} \ar[r]_{\nu}     & G^{2n} \ar[r] \ar[d]_{\cong} & \HH_1(G, \ZZ) \ar[r] \ar@{=}[d] & 1 \\
             1 \ar[r] & \HH_2(G, \ZZ) \ar[r]_{\alpha} & \pi_1(F_{h_n})           \ar[r]  & \pi_1(E_{h_n}) \ar[r]                       & \HH_1(G, \ZZ) \ar[r] & 1.}\]
We already know by Corollaries~\ref{cor:main} and~\ref{cor:ker-iso} that $K$ is naturally isomorphic to $\HH_2(G, \ZZ)$ provided $n\geq 3$.
Thus, it suffices to show that $\psi$ is an isomorphism. In view of~\cref{lm:endotr} we only need to consider the case when $G$ is abelian.

Now we compute the image of each generator $c_{ij}(x, y) = x \wedge y \in \HH_2(G, \ZZ)$ under $\psi$.
In order to do this we need a more explicit description of the set of 1-simplices of $F_{h_n}$.
Since both $BG^{2n}$ and $\ZZ[BG]$ are fibrant $E_{h_n}$ and $F_{h_n}$ can be constructed using the path space construction, i.\,e. $E_{h_n} = BG^{2n} \times_{\ZZ[BG]} \ZZ[BG]^I$.
In this case the set of $1$-simplices of $F_{h_n}$ can be identified with the set of triples $(g, t, t') \in G^{2n}\times \ZZ[BG]_2 \times \ZZ[BG]_2$ 
 satisfying the following identites:
\[d_2(t) = h_n(g),\ \ d_1(t) = d_1(t'),\ \ d_0(t') = 0.\]
It will be convenient for us to perform all calculations inside the subset $F'\subset (F_{h_n})_1$ consisting of those triples for which $t'=0$ and $d_1(t)=d_0(t)=0$.
For shortness we write down an element of $F'$ as $(g, t)$.
It is easy to check that the image of the loop $p_{ij}(x)$ under the natural map $T \to F_{h_n}$ corresponds to the pair
$((e, d_{ij}(x)), t(x)),$ where $t(x)=-(x, x^{-1}) + (e, e) - (x^{-1}, e) + (e, x^{-1})$.

We want to devise a concrete formula for the element of $F'$ homotopic to the concatenation of two elements $(g_1, t_1), (g_2, t_2) \in F'$.
It is easy to check that $(g_1g_2, t)$, where $t$ is the $1$-st face of the filler $f$ for the $3$-horn $(t_2, ,t_1, h_n(g_1, g_2))$, is the desired element.
By Moore's theorem one can find a filler for any horn $(p_0, ,p_2, p_3)$ in a simplicial abelian group via the formula
 $f = s_0 p_0 - s_0 s_1 d_2 p_0 - s_0 s_0 d_1 p_0 + s_0 s_0 d_2 p_0 + s_1 p_2 + s_2 p_3 - s_1 p_3$. 
Substituting concrete values of $p_i$ into this formula and using the fact that $d_1t_2=0$ we get the following expression for $t$:
\begin{multline} \nonumber t = d_1f = p_0 - s_1 d_2 p_0 - s_0 d_1 p_0 + s_0 d_2 p_0 + p_2 - p_3 + s_1 d_1 p_3 = \\
             = t_1 + t_2 - h_n(g_1, g_2) - s_1 h_n(g_2) + s_0 h_n(g_2) + s_1 h_n (g_1g_2). \end{multline}
Applying the above formula twice we get that the image of $c_{ij}(u,v)$ under $\varphi\beta$ equals $((e, e), t)$, where
\begin{multline} \label{eq:formula-t} t = 2(e, e)-(x^{-1}, e) +(xy, x^{-1}y^{-1}) -(xy, e) -(x, x^{-1}) -(y, y^{-1}) -\\  -(e, x^{-1}y^{-1}) +(e, x^{-1}) +(x, y) +(x^{-1}, y^{-1}) -(e, y) -(xyx^{-1}y^{-1}, e) +(y, e).\end{multline}

The map $\alpha$ can be described as the Dold---Kan isomorphism $\HH_2(G, \ZZ) \cong \pi_2(\ZZ[BG])$ followed by the boundary map $\partial\colon \pi_2(\ZZ[BG]) \to \pi_1(F_{h_n})$.
It is not hard to show that $\partial$ maps each $c \in \pi_2(\ZZ[BG])$ to $((e,e),-c)$ (see~\cite[p.~29]{GJar09}).
Thus, we get that $\psi$ maps $x \wedge y$ to $-t$, where $t$ is given by~\eqref{eq:formula-t}.

On the other hand, the generator $x \wedge y$ of $\HH_2(G, \ZZ)$ corresponds to the 2-cycle $c = (x, y) - (y, x)$ (cf.~\cite[(14), p.~582]{Mi52}) and the corresponding
element of $\pi_2(\ZZ[BG])$ is simply the normalized 2-cycle of $c$:
\begin{equation} \nonumber c' = c - s_0d_0c - s_1d_1c + s_1d_0 = (x, y) - (y, x) - (x, e) +(y, e) + (e, x) - (e, y). \end{equation}
It it not hard to check that $t-c'$ is the image of the following $3$-chain under the differential $d_0-d_1+d_2-d_3$:
\[(xy, x^{-1}, y^{-1}) - (y, x, x^{-1}) + s_1 s_1(x^{-1} + xy - x - y) + s_0 s_0(x^{-1} - x - x^{-1}y^{-1}).\]
We obtain that the classes of $t$ and $c'$ in $\HH_2(G, \ZZ)$ are the same, hence $\psi$ is an isomorphism, as claimed.
\end{proof}

\printbibliography

\end{document}